\documentclass{article}

\usepackage[utf8]{inputenc}
\usepackage[T1]{fontenc}

\usepackage{amssymb}
\usepackage{amsmath}
\usepackage{amsfonts}
\usepackage{latexsym}
\usepackage{amsthm}
\usepackage{mathrsfs} 
\usepackage[colorlinks, urlcolor=blue, citecolor=black, linkcolor=black, breaklinks]{hyperref} 

\usepackage{enumerate}

\newtheorem{Definition}{Definition}[section]
\newtheorem{Theorem}[Definition]{Theorem}
\newtheorem{Lemma}[Definition]{Lemma}
\newtheorem{Remark}[Definition]{Remark}
\newtheorem{Corollary}[Definition]{Corollary}
\newtheorem{Proposition}[Definition]{Proposition}

\def\Z{\mathbb{Z}}

\def\F{\mathbb{F}}
\def\P{\mathbb{P}}

\newcommand{\Ld}[2][D]{\ensuremath{\mathcal{L}(#2\mathcal{#1})}}

\newcommand\subrel[2]{\mathrel{\mathop{#2}\limits_{#1}}}
\newcommand{\udots}{\mathinner{\mskip1mu\raise1pt\vbox{\kern7pt\hbox{.}}
  \mskip2mu\raise4pt\hbox{.}\mskip2mu\raise7pt\hbox{.}\mskip1mu}}

\title{On the construction of elliptic Chudnovsky-type algorithms for multiplication in large extensions of finite fields.}

\author{
{Stéphane Ballet{\small $~^{1}$}, Alexis Bonnecaze{\small $~^{2}$}, Mila Tukumuli{\small $~^{3}$}}\\
Aix Marseille University\\Institut de Mathématiques de Luminy\\
Campus de Luminy\\
13288 Marseille Cedex 9, France\\
}

\date{\today}

\begin{document}
 
\maketitle

\begin{abstract}
\noindent
We indicate a strategy in order to construct bilinear multiplication algorithms 
of type Chudnovsky in large extensions of any finite field. In particular, 
by using the symmetric version of the generalization of Randriambololona specialized 
on the elliptic curves, we show that it is possible to construct such algorithms with low bilinear complexity.  
More precisely, if we only consider the Chudnovsky-type algorithms of type symmetric elliptic, we show 
that the symmetric bilinear complexity of these algorithms is in $O(n(2q)^{\log_q^*(n)})$ where $n$ corresponds to 
the extension degree, and $\log_q^*(n)$ is the iterated logarithm. Moreover, we show that the construction of such algorithms can be done in time polynomial in $n$. Finally, applying this method we present the effective construction, step by step, of such an algorithm of multiplication in the finite field $\F_{3^{57}}$.
\end{abstract}


\section{Introduction}
\noindent
A growing number of applications, such as asymmetric cryptography, make use of big integer arithmetic. In this context, it is important to conceive and develop efficient arithmetic algorithms combined with an optimal implementation method. Accelerating basic arithmetic operations can provide efficient arithmetic algorithms and thus, can make faster a protocol which executes a lot of multiplications. This situation typically occurs when considering cryptographic protocols. In this paper, we only care about the multiplication operation. There exist numerous multplication algorithms in the literature, examples are Karatsuba's algorithm for polynomial multiplication, Toom-Cook's algorithm for large integer multiplication, but also Strassen's algorithm for matrix multiplication.  In this article, we are interested in multiplication algorithms in any extension of finite fields, in particular the focus is on the Chudnovsky-Chudnovsky method~\cite{Chudnovsky}. This method, based on interpolation on algebraic curves defined over a finite field  allows us to obtain multiplication algorithms with low bilinear complexity. Our objective is to construct explicitely such multiplication algorithms for large finite extensions of finite fields. The Chudnovsky-Chudnovsky method and its variants have been extensively studied these last years through the work of Shparlinsky, Tsfasmann, Vladut~\cite{shtsvl}, Baum and Shokrollahi~\cite{ShokBaum}, Ballet- and Rolland~\cite{Ballet4},~\cite{Ballet5}, 
Chaumine~\cite{chau}, Arnaud~\cite{Arnaud}, Cenk-Ozbudak~\cite{CenkOzbudak} and Cascudo, Cramer, Xing and Yang~\cite{cacrxiya}, and recently Randriambololona~\cite{Randriam}. Indeed, the studies on the subject are of both theoretical and practical importance: theoretically, the bilinear complexity is linked to the tensor rank and in practice, it is related to the number of gates in an electronic circuit. 
However, most of the work focused on the improvement of the bounds on the bilinear complexity and the theoretical aspects of the Chudnovsky-type algorithms (in particular the underlying geometry of Riemann-Roch spaces).

\subsection{Multiplication algorithm and tensor rank}

Let  $\F_{q}$ be a finite field where $q$ is a prime power, $\F_{q^n}$ is the degree $n$ extension of $\F_{q}$ and $(e_{1},\ldots, e_{n})$ denotes a basis of $\F_{q^n}$ over $\F_{q}$. We define for two elements of $\F_{q^n}$
$$ X=\sum_{i=1}^{n}x_{i}e_{i} \hspace{1cm} \mbox{and} \hspace{1cm} Y=\sum_{i=1}^{n}y_{i}e_{i}, $$   
\noindent
the complexity of multiplication of $\F_{q^n}$ over $\F_{q}$ as the number of elementary operations needed to obtain the product $X.Y$ in $\F_{q^n}$, where elementary operations are:
\begin{enumerate}
\item
 Addition: $ (a,b) \mapsto a + b $,~~with~~ $a,b~\in~\F_{q}$.
\item 
 Scalar multiplication: $ x \mapsto c.x $,~~with~~ $c,x~\in~\F_{q}$.
\item 
Bilinear multiplication: $ (a,b) \mapsto a. b $,~~with~~ $a,b~\in~\F_{q}$.
\end{enumerate} 
\noindent
In this paper, we focus on the construction of algorithms realizing the  multiplication in extensions of finite fields 
with a minimal number (called bilinear complexity) of two-variable multiplications (called bilinear multiplications) without considering 
the other operations as multiplications by a constant (called scalar multiplications). More precisely, let us recall the notions 
of multiplication algorithm and associated bilinear complexity in terms of tensor rank.  


\begin{Definition}
Let $K$ be a field, and $E_0, \ldots, E_s$ be finite dimensional $k$-vector spaces. A non zero element $t\in E_0 \otimes \cdots \otimes E_s$ is said to be an elementary tensor, or a tensor of rank 1, if it can be written in the form $t=e_0 \otimes \cdots \otimes e_s$ for some $e_i \in E_i$. More generally, the rank of an arbitrary $t\in E_0 \otimes \cdots \otimes E_s$  is defined as the minimal length of a decomposition of $t$ as a sum of elementary tensors.
\end{Definition}

\begin{Definition}
If $$ \alpha~~:~~E_1\times \cdots \times E_s \longrightarrow E_0$$
is an $s$-linear map, the $s$-linear complexity of $\alpha$ is defined as the tensor rank of the element
$$\tilde{\alpha}\in E_0\otimes E_1^{\vee} \otimes \cdots \otimes E_s^{\vee}$$ where $E_i^{\vee}$ denotes the dual of $E_i$ as vector space over $K$ for any integer $i$, naturally deduced from $\alpha$. In particular, the $2$-linear complexity is called the bilinear complexity.
\end{Definition}

\begin{Definition}
Let $\mathcal A $ be a finite-dimensional $K$-algebra. We denote by 
$$\mu(\mathcal A/K)$$
the bilinear complexity of the multiplication map
$$m_{\mathcal A}~~:~~\mathcal A \times \mathcal A \longrightarrow \mathcal A$$
considered as a $K$-bilinear map. 

\noindent
In particular, if $\mathcal{A}=\F_{q^n}$ and $K=\F_q$,  we let: $$\mu_q(n)=\mu(\F_{q^n}/\F_q).$$
\end{Definition}

\noindent
More concretely, $\mu(\mathcal A/K)$ is the smallest integer $n$ such that there exist linear forms $\phi_1,\ldots, \phi_n$ and $\psi_1, \ldots, \psi_n~~:~~\mathcal A \longrightarrow K$, and elements $w_1, \ldots, w_n \in \mathcal A$, such that for all $x,y\in \mathcal A$ one has
\begin{equation} xy= \phi_1(x)\psi_1(y)w_1+\cdots +\phi_n(x)\psi_n(y)w_n, \label{xy}\end{equation}
since such an expression is the same thing as a decomposition

\begin{equation}
T_M=\sum_{i=1}^{n}w_i\otimes \phi_i\otimes \psi_i \in  \mathcal{A} \otimes \mathcal{A} \otimes \mathcal{A}^{\vee}.
\end{equation}
\noindent
for the multiplication tensor of $\mathcal{A}$.

\begin{Definition}
We call multiplication algorithm of length $n$  for $\mathcal A/K$ a collection of $\phi_i, \psi_i, w_i$ that satisfy \eqref{xy} or equivantly a tensor decomposition $$T_M=\sum_{i=1}^{n}w_i\otimes \phi_i\otimes \psi_i \in  \mathcal{A} \otimes \mathcal{A} \otimes \mathcal{A}^{\vee}$$ 
for the multiplication tensor of $\mathcal{A}$. Such an algorithm is said symmetric if $\phi_i=\psi_i$ for all $i$ (this can happen only if $\mathcal A$ is commutative).
\end{Definition}
\noindent
Hence, when $\mathcal{A}$ is commutative, it is interesting to study the minimal length of a symmetric multiplication algorithm. 

\begin{Definition}
If $\mathcal A $ is a finite-dimensional $K$-algebra. The symmetric bilinear complexity $$\mu^{sym}(\mathcal A/K)$$ 
is the minimal length of a symmetric multiplication algorithm. 

\noindent
In particular, if $\mathcal{A}=\F_{q^n}$ and $K=\F_q$,  we let: $$\mu^{sym}_q(n)=\mu^{sym}(\F_{q^n}/\F_q).$$
\end{Definition}

\subsection{Known results}
Let us recall some classical known results. In their seminal papers, Winograd~\cite{Winograd} and De Groote~\cite{Groote} have shown that $ \mu(\F_{q^n}/ \F_{q}) \geq 2n-1 $, with equality holding if and only if $n \leq \frac{1}{2}q + 1$. Winograd have also proved~\cite{Winograd} that optimal multiplication algorithms realizing the lower bound belong to the class of interpolation algorithms. Later, generalizing  interpolation algorithms on the projective line over $\F_{q}$ to algebraic curves of higher genus over $\F_{q}$, Chudnovsky and Chudnovsky provided a method~\cite{Chudnovsky} which enabled to prove the linearity~\cite{Ballet2} of the bilinear complexity of multiplication in finite extensions of a finite field. Moreover, they proposed the first known  multiplication algorithm using interpolation to algebraic function fields (of one variable) over $ \F_{q} $.  This is the so-called Chudnovsky and Chudnovsky algorithm, also called Chudnovsky algorithm to simplify. Then, several studies will focus on the qualitative improvement of this algorithm (for example~\cite{Ballet4},~\cite{Arnaud},~\cite{CenkOzbudak},~\cite{Randriam}) as well as the improvements of upper bounds (for example~\cite{Ballet5},~\cite{BalletPieltant}) and asymptotic upper bounds (for example~\cite{shtsvl},~\cite{cacrxiya}) of the bilinear complexity. However, few studies have been devoted to the effective construction of Chudnovsky-type algorithms, and in particular no work has been done when the degree of extensions reach cryptographic size. Indeed, the first known effective finite fields multiplication through interpolation on algebraic curves was proposed by Shokrollahi and Baum ~\cite{ShokBaum}. They used the Fermat curve $ x^3 + y^3 = 1$ to construct multiplication algorithm over $\F_{4^4}$ with 8 bilinear multiplications. In~\cite{Ballet3}, Ballet proposed one over $ \F_{16^n} $ where $ n~\in~[13,14,15]$, using the hyperelliptic curve $y^2 + y = x^5$ with $2n + 1$  bilinear multiplications. Notice that these aforementioned two algorithms only used rational points, and multiplicity equals to one. 
Recently Cenk and \"{O}zbudak proposed in~\cite{CenkOzbudak} an explicit multiplication algorithm in $ \F_{3^9} $ with 26 bilinear multiplications. To this end, they used the elliptic curve $ y^2 = x^3 +  x  + 2 $ with points of higher degree and higher multiplicity.

\subsection{Organization of the paper and new results}

In Section 2, we fix the notation and we recall the different versions of Chudnovsky-type algorithms. 
Then in Section 3, we present a strategy in order to construct multiplication algorithms 
of type Chudnovsky in arbitrary large extensions of finite fields. 
In particular, we show that from an elliptic curve defined over any finite field $\F_q$, we can exhibit a 
symmetric version of the generalization of Randriambololona (specialized 
on the elliptic curves) for any extension of $\F_q$ of degree $n$, with low bilinear complexity.  
More precisely, if we only consider the Chudnovsky-type  algorithms of type symmetric elliptic, we show 
that the symmetric bilinear complexity of these algorithms is in $O(n(2q)^{\log_q^*(n)})$. 
Even if this asymptotical complexity is quasi-linear, it has the advantage to be derivated from an infinite family 
of symmetric algorithms with a fixed genus equals to one, which corresponds to the specificity of our strategy in contrast 
to the usual strategy.  Consequently, fixing the genus to one allows us to control the complexity of the construction, 
meaning that for finite fields of cryptographic size, one can construct in a reasonable time such algorithms. 
Indeed, we prove that the complexity of the construction of symmetric elliptic algorithms is in time polynomial in $n$. 
In fact, this is not at all the case of the usual strategy based upon the construction of algorithms with growing genus 
since the complexity of such a construction is not known because of the problem of the explicit construction of 
high degree points~\cite[Section 4, Remarks 5]{shtsvl}. Finally in section 4, we present new upper bounds for 
large extensions of $\F_2$ and $\F_3$, and we also propose the effective construction, 
step by step, of an algorithm of multiplication in $\F_{3^{57}}$.

\section{Multiplication algorithms of type Chudnovsky}

We start with some elementary terminology and results of algebraic function fields. 
A comprehensive course of the subject  can be found in~\cite{Stich}.

\subsection{Notation}

An algebraic function field $F /\F_{q} $ of one variable over $\F_{q}$ is an extension field $ F \supseteq \F_{q}$ such that F is a finite extension of $\F_{q} (x)$ for some element $ x ~ \in ~ F $ which is transcendental over $\F_{q}$. A valuation ring of the function field $F /\F_{q}$ is a ring $\mathcal{O} \subseteq F$ such that $\F_{q} \subset \mathcal{O} \subset F$ and for any $z \in F$, either $ z \in \mathcal{O} ~or~ z^{-1} \in \mathcal{O}$. A place P of the function field $F /\F_{q}$ is the maximal ideal of some valuation ring $\mathcal{O}$ of $F /\F_{q}$. If $\mathcal{O}$ is a valuation ring of $F /\F_{q}$ and P is its maximal ideal, then $\mathcal{O}$ is uniquely determined by P hence we denote $\mathcal{O}$ by $\mathcal{O}_{P}$.
Every place $P$ can be written as $P = t\mathcal{O}_{P}$, where $t$ is the local parameter for $P$. We will denote the set of all places of $F /\F_{q}$ as $\P_{F}$. For a place $P$,
$ F_{P} := \mathcal{O}_{P}/P $ is called the residue class field of P. The map $ x \rightarrow x(P)$ from $F$ to $F_{P} \cup \{\infty\}$ is called the residue class map with respect to P. The degree of P is defined by $ [F_{P} : \F_{q} ] := \deg P $. The free abelian group which is generated by the places of $F /\F_{q}$ is called the divisor group of $F /\F_{q}$ and it is denoted by $\mathscr{D}_{F}$, so a divisor is a formal sum $D = \sum_{P \in \P_{F}}n_{P}P $, with $n_{P} \in \Z$ almost all $n_{P}= 0 $, of degree $\deg (D) = \sum_{P \in \P_{F}} v_{P}(D).\deg P $ where $v_{P}$ is a discrete valuation associated to the place $P$. The support of a divisor $D$ denoted $supp~D$ is the set of places $P$ with $v_{P}(D) \neq 0$. For a function $f \in F/\F_{q}$, we denote by $(f) = \sum_{P \in \P_{F}} v_{P}(f).P$ the principal divisor of $f$. If $D$ is a divisor then $\mathscr{L}(D) = \left\{ f~\in~ F ~|~D + (f) \geq 0 \right\} \cup \left\lbrace 0\right\rbrace  $ is the Riemann-Roch space which is a  $\F_{q}$-vector space. The integer $ \ell(D) = \dim \mathscr{L}(D)$ is called the dimension of $D$ and $i(D) = \dim D - \deg D + g - 1$
is the index of speciality of $D$. We say that $D$ is non-special if $i(D) = 0$ and special otherwise.\\

\subsection{Original Algorithm of Chudnovsky}

We are now able to state the original Chudnovsky algorithm~\cite{Chudnovsky} and its recent improvements.

\begin{Theorem}\label{AlgoChud}
Let
\begin{itemize}
\item $F/\F_{q}$ be an algebraic function field of one variable,
\item $Q$ be a degree n place of $F/\F_{q}$,
\item $D$ be a divisor of $F/\F_{q}$,
\item $P = \{P_{1}, \ldots, P_{N} \}$ be a set of rational places.
\end{itemize}
We suppose that $ supp~D \cap  \{Q, P_{1}, \ldots, P_{N} \} = \emptyset $  and that:

$A-$ The application
 
                                  $$   Ev_{Q} : ~~~\mathscr{L}(D) \longrightarrow  \frac{\mathcal{O}_{Q}}{Q} $$
$$ \hspace{2cm}   f ~~\longmapsto  f(Q) $$

is surjective. 

$B-$ The application
\begin{center} 
$ Ev_{\mathcal{P}}:~~\mathscr{L}(2D)\longrightarrow  \F_{q}^{N} $\\
$ \hspace{4,8cm} f ~~\longmapsto  \left(f(P_{1}),\ldots,f(P_{N}) \right) $
\end{center}       
is injective.

Then
                $$ \mu^{sym}({\F_{q^n}/\F_{q}}) \leq N.$$

\end{Theorem}

\noindent
A drawback of this algorithm is that it only uses rational points. Moreover, finding sufficiently rational points and suitable divisors such that evaluation maps are surjective and injective is either a difficult task, or even impossible. Consequently, some researchers proposed several improvements and variants that we present in the next section. 

\vspace{1em}

\subsection{Generalization of Arnaud and Cenk-Ozbudak}

In order to obtain good estimates for the bilinear complexity, S. Ballet gave in~\cite{Ballet2} some conditions easy to verify allowing the use of Chudnovsky  algorithm. Then S. Ballet and R. Rolland generalized in~\cite{Ballet4} the original algorithm using places of degree $1$ and~$2$. The best finalized version of this algorithm in this direction, is the generalization introduced by N. Arnaud in~\cite{Arnaud}  and improved by M. Cenk and F. \"Ozbudak in~\cite{CenkOzbudak}. This generalization uses several coefficients in the local expansion at each place $P_i$ instead of just the first one. Due to the way to obtain the local expansion of a product from the local expansion of each term, the bound for the bilinear complexity involves the complexity notion $\widehat{M_q}(u)$ introduced by Cenk and \"Ozbudak in~\cite{CenkOzbudak} and defined as follows:
\begin{Definition}\label{DefMu}
We denote by $\widehat{M_q}(u)$ the minimum number of multiplications needed in $\F_q$ in order to obtain coefficients of the product of two arbitrary $u$-term polynomials modulo $x^u$ in $\F_q[x]$.
\end{Definition}
\noindent
For instance, we know that for all prime powers $q$, we have $\widehat{M_q}(2) \leq 3$ by~\cite{CenkOzbudak2}. Now, we introduce the generalized algorithm of type Chudnovsky described in~\cite{CenkOzbudak}.

\begin{Theorem} \label{theo_evalder}
Let \\
\vspace{.1em}
$\bullet$ $q$ be a prime power,\\
\vspace{.1em}
$\bullet$ $F/\F_q$ be an algebraic function field,\\
\vspace{.1em}
$\bullet$ $Q$ be a degree $n$ place of $F/\F_q$,\\
\vspace{.1em}
$\bullet$ ${\mathcal D}$ be a divisor of $F/\F_q$,\\
\vspace{.1em}
$\bullet$ ${\mathcal P}=\{P_1,\ldots,P_N\}$ be a set of $N$ places of arbitrary degree,\\
\vspace{.1em}
$\bullet$ $u_1,\ldots,u_N$ be positive integers.\\
We suppose that $Q$ and all the places in $\mathcal P$ are not in the support of ${\mathcal D}$ and that:
\begin{enumerate}[a)]
	\item the map
	$$
	Ev_Q:  \left \{
	\begin{array}{ccl}
	\Ld{} & \rightarrow & \F_{q^n}\simeq F_Q\\
	f & \longmapsto & f(Q)
	\end{array} \right.
	$$ 
	is onto,
	\item the map
	$$
	Ev_{\mathcal P} :  \left \{
	\begin{array}{ccl}
	\Ld{2} & \longrightarrow & \left(\F_{q^{\deg P_1}}\right)^{u_1} \times \left(\F_{q^{\deg P_2}}\right)^{u_2} \times \cdots \times \left(\F_{q^{\deg P_N}}\right)^{u_N} \\
	f & \longmapsto & \big(\varphi_1(f), \varphi_2(f), \ldots, \varphi_N(f)\big)
	\end{array} \right.
	$$
	is injective, where the application $\varphi_i$ is defined by
	$$
	\varphi_i : \left \{
	\begin{array}{ccl}
	\Ld{2} & \longrightarrow & \left(\F_{q^{\deg P_i}}\right)^{u_i} \\
          f & \longmapsto & \left(f(P_i), f'(P_i), \ldots, f^{(u_i-1)}(P_i)\right)
	\end{array} \right.
 	$$
	with $f = f(P_i) + f'(P_i)t_i + f''(P_i)t_i^2+ \ldots + f^{(k)}(P_i)t_i^k + \ldots $, 
the local expansion at $P_i$ of $f$ in ${\Ld{2}}$, with respect to the local parameter~$t_i$. 
Note that we set ${f^{(0)} =f}$.
\end{enumerate}
Then 
$$
\mu^{sym}_q(n) \leq \displaystyle \sum_{i=1}^N \mu^{sym}_q(\deg P_i) \widehat{M}_{q^{\deg P_i}}(u_i).
$$
\end{Theorem}

\noindent
Remark that the original algorithm in~\cite{Chudnovsky} given by D.V. and G.V. Chudnovsky regards the particular case $\deg P_i=1$ and $u_i=1$ for $i=1, \ldots, N$. The first generalization introduced by S.Ballet and R. Rolland in~\cite{Ballet4} allows the use of place of degree one and two, more precisely it concerns the case $\deg P_i=1 \hbox{ or }2$ and $u_i=1$ for $i=1, \ldots, N$. Next,  N. Arnaud introduced during his PhD~\cite{Arnaud}, the use of derivative evaluation which provides refinement of bounds of bilinear complexity. His work concerns the case $\deg P_i=1 \hbox{ or }2$ and $u_i=1\hbox{ or }2$  for $i=1, \ldots, N$. Cenk and \"Ozbudak generalized in~\cite{CenkOzbudak} Arnaud's work not only interpolating on places of arbitrary degree but also using derivative evaluation as desired. Thus less places of fixed degree are necessary to get the injectivity and the surjectivity of both evaluation maps. However, they use separately the degree $\deg P_{i}$ of a place $ P_{i}$ and its multiplicity $ \widehat{M}_{q}(u_{i})$. Recently, Randriambololona introduced in~\cite{Randriam} a new generalization of this algorithm which combine them.

\subsection{Generalization of Randriambololona}\label{GeneRandriam}

Randriambololona introduced in~\cite{Randriam} a possibly asymmetric version of this algorithm. Furthermore, he introduced a new quantity $ \mu_{q}(\deg P_{i},u_{i})$ to deal with both, the degree and the multiplicity, at the same time.

\begin{Definition}
For any integers $n,l\geq 1$ we consider the $\F_q$-algebra of polynomials in one indeterminate with coefficients in $\F_{q^n}$, truncated at order $l$:
$${\mathcal A}_q(n,l) = \F_{q^n}[t]/(t^l)$$
of dimension
$$dim_{\F_q}\mathcal A_q(n,l) = nl,$$
and we denote by 
$$\mu_q(n,l) = \mu(\mathcal A_q (n,l)/\F_q)$$
its bilinear complexity over $\F_q$ and by 
$$\mu^{sym}_q(n,l) = \mu^{sym}(\mathcal A_q (n,l)/\F_q)$$ 
its symmetric bilinear complexity over $\F_q$. 
\end{Definition}

\noindent
Note that when $l=1$, we have $\mu_q(n,1)=\mu_q(n)$ which corresponds to the bilinear complexity of multiplication in $\F_{q^n}$ over $\F_q$; 
and when $n=1$, we have $\mu_q(1,l)=\widehat{M}_{q^{\deg P_i}}(l)$ which represents the quantity defined by Cenk and Ozbudak~\cite{CenkOzbudak}.
Now, in order to make easier the presentation of Randriambololona's generalization, we choose to use the language of modern algebraic geometry 
emphasizing the geometric point of view even if everything could be equally expressed in the language of function fields in one indeterminate. 
Hence, by a thickened point in the algebraic curve $X$ defined over $\F_q$, we mean any closed subscheme of $X$ supported on a closed 
point (of arbitrary degree). If $Q$ is a closed  point in $X$, we denote by $\mathcal{I}_Q$ the sheaf of ideals defining it and for any integer 
$l\geq 1$, we let $Q^{[l]}$ be the closed subscheme of $X$ defined by the sheaf of ideals $(\mathcal{I}_Q)^l$. Then $Q^{[l]}$  
is the thickened point supported on $Q$. If $D$ is a divisor on $X$, we denote by $\mathcal{L}(D)=\Gamma(X,\mathcal{O}_X(D))$ 
its Riemann-Roch space. Then, we can present the generalization in~\cite{Randriam} which corresponds to the asymmetric version of algorithm of type Chudnovsky.

\begin{Theorem}\label{AlgoRandriam}
Let $\mathcal{C}$ be a curve of genus $g$ over $\F_q$, and let $n,l\geq 1$ be two integers. Suppose that $\mathcal{C}$ admits a closed point $Q$ of degree $\deg Q=n$. Let $G$ be an effective divisor on $\mathcal{C}$, and write 
$$ G=u_1P_1+\cdots +u_NP_N $$
where the $P_i$ are pairwise distinct closed points, of degree $\deg P_i=d_i$. Suppose there exist two divisors  $D_1,D_2$ on $\mathcal{C}$ such that:
\begin{enumerate}
\item[(i)] The natural evaluation  map
$$\mathcal{L}(D_1+D_2) \longrightarrow \prod_{i=1}^N \mathcal{O}_{\mathcal{C}}(D_1+D_2)\mid_{P_i^{[u_i]}}
$$
is injective.
\item[(ii)] The natural evaluation maps 
$$\mathcal{L}(D_1)  \longrightarrow \mathcal{O}_{\mathcal{C}}(D_1)\mid_{Q^{[l]}}~~~~~~~
\mathcal{L}(D_2)  \longrightarrow \mathcal{O}_{\mathcal{C}}(D_2)\mid_{Q^{[l]}}
$$
are surjective.
\end{enumerate}
Then $$\mu_{q}(n,l) \leq \sum_{i=1}^{N} \mu_{q}(d_{i},u_{i}). 
$$
In fact, we also have $\mu_q(n,l)\leq \mu(\prod_{i=1}^N \mathcal{A}_q(d_i,u_i)/\F_q)$. Moreover, if $D_1=D_2$, all these inequalities also hold for the symmetric bilinear complexity $\mu^{sym}$.\\
Sufficient numerical criteria for the hypotheses above to hold can be given as follows. A sufficient condition for the existence of $Q$ of degree $n$ on $\mathcal{C}$ is that 
$2g+1 \leq q^{(n-1)/2}(q^{1/2}-1)$, while sufficient conditions for $(i)$ and $(ii)$ are:
\begin{enumerate}
\item[(i')]  The divisor $D_1+D_2-G$ is zero-dimensional:
$$ l(D_1+D_2-G)=0.$$
\item[(ii')] The divisors $D_1-lQ$ and $D_2-lQ$ are non-special:
$$i(D_1-lQ)=i(D_2-lQ)=0.$$
\end{enumerate}
More precisely, $(i)$ and $(i')$ are equivalent, while $(ii')$ only implies $(ii)$ a priori.
\end{Theorem}
\vspace{1em}

\noindent
The improvement suggested by Randriambololona in relation with bilinear complexity leads to the following inequality
$$ \mu_{q}(\deg P_{i},u_{i}) \leq  \mu_{q}(\deg P_{i})\widehat{M}_{q^{\deg P_{i}}}(u_{i})),$$
\noindent 
where  $ \mu_{q}(\deg P_{i},1) =  \mu_{q}(\deg P_{i})$  is the bilinear complexity of multiplication in $\F_{q^{\deg P_{i}}}$ over $\F_{q}$, and $ \mu_{q}(1,u_{i}) = \widehat{M}_{q^{\deg P_{i}}}(u_{i})$  is the complexity previously defined in Definition \ref{DefMu}. There exist examples where this inequality is strict, especially when we use places of higher degree with higher multiplicity. It is not the case in this paper. In fact, even if the formula $\mu_{q}(\deg P_{i},u_{i})$ is recursive, meaning that we can derive upper bounds, using  places of higher degree with higher multiplicity is more expensive than only use higher multiplicity with rational places.\\  

\begin{table}[h]
\begin{center}
\caption{\label{T3}Bounds for $ \mu_{q}(n)$ and $~\widehat{M}_{q}(n)$ for 1 $ \leq n \leq $ 8, and $q=2,3$.}
~~\\
\begin{tabular}{|c|c|c|c|c|c|c|c|c|}
\hline
$n$ & ~1~ &~ 2~ &~ 3~ &~ 4 ~& ~5~ &~ 6~ & ~7~ &~ 8 ~\\
\hline
$~\mu_{2}(n,1)=\mu_{2}(n)~ $ & 1 & 3 & 6 & 9 & 13 & 15 & 22 & 24 \\
\hline
$~\mu_{3}(n,1)=\mu_{3}(n)~ $ & 1 & 3 & 6 & 9 & 12 & 15 & 19 & 21\\
\hline   
$~\widehat{M}_{q}(n)~ $ & 1 & 3 & 5 & 8 & 11 & 15 & 19 & 24 \\ 
\hline
\end{tabular}
\end{center}
\end{table}
\vspace{1em}

\section{Construction of certain algorithms of type Chudnovsky}

\subsection{Strategies of construction}

So far, the strategy to obtain upper bounds for bilinear complexity of multiplication in $\F_{q^n}$ over $ \F_q$, has always been to apply algorithms of type Chudnovsky on infinite families (specially some towers) of algebraic function fields defined over a fixed finite field $\F_q$, with genus growing to infinity. More precisely, from a practical point of view, for any integer n, it consists in choosing the appropriate algebraic function field in the family, namely the first one satisfying the  conditions of Theorems~\ref{AlgoChud},~\ref{theo_evalder} or~\ref{AlgoRandriam}, in order to multiply in $\F_{q^n}$. This implies increasing the genus for few fixed degrees of places. Unfortunately this strategy has a weak point since growing the genus could hugely increase the complexity of the construction. However, there exists another strategy which corresponds to using the degree of freedom that remains: the degree of places. Technically, this approach consists in fixing the genus while increasing the degree of places. This new way, implied in the generalization of Arnaud and Cenk-Ozbudak, has never been investigated and requires introducing new complexity notions.

\begin{Definition}
For any integers $n,l\geq 1$, and for the $\F_q$-algebra $${\mathcal A}_q(n,l) = \F_{q^n}[t]/(t^l),$$ let us set
                     $$\mu_{q,g}(n,l) = ~\subrel{C} \min \mu(\mathcal A_q (n,l)/\F_q),$$
where $C$ is running over all curves of genus $g$ over $\F_q$. Then $\mu_{q,g}(n,l)$ is called the bilinear complexity over $\F_q$ of  the $\F_q$-algebra ${\mathcal A}_q(n,l)$ when the genus $g$ is fixed.  
\noindent
We denote $$\mu_{q,\mathcal{C}}(n,l) = \mu(\mathcal A_q (n,l)/\F_q),$$
the bilinear complexity over $\F_q$ of the $\F_q$-algebra ${\mathcal A}_q(n,l)$ when the model of the curve of genus $g$ is fixed. Quantities $\mu^{sym}_{q,g}(n,l) $ and $\mu^{sym}_{q,\mathcal{C}}(n,l) $ denote their associated symmetric bilinear complexity over $\F_q$. 
\end{Definition}

\noindent
Our purpose here is to develop this strategy in the case of elliptic function fields. The choice of algebraic curves of genus one was made for two main reasons:
\begin{enumerate}                                                                                                                     \item First of all, because the effective construction of such elliptic algorithms can be completed within a reasonable time. More precisely, we prove that the complexity of the construction of a symmetric elliptic bilinear multiplication algorithm in $\F_{q^n}$  is in  time polynomial in $n$.
\item Finally, elliptic curves are heavily used to construct cryptographic primitives. Indeed, using the same elliptic curve for both the multiplication and the cryptographic algorithms could improve the efficiency in secure embedded systems.
\end{enumerate}

\subsection{Elliptic Chudnovsky algorithms}

In this section, we improve a result obtained by Randriambololona in~\cite[Proposition 4.3]{Randriam} which, 
setting the parameter $\ell$ to 1, generalizes a result of Shokrollahi~\cite{Shokrollahi} and Chaumine~\cite{chau}.\\  
\noindent
Let $\mathcal{C}/\F_q$ be an elliptic curve defined over $\F_q$ with a chosen point $P_{\infty}$. The set $\mathcal{C}(\F_q)$ of rational points over $\F_q$  admits a structure of finite abelian group with identity element $P_{\infty}$ and a cardinal $N_1(\mathcal{C}(\F_q))$. 
Moreover, there is a map $ \sigma : Div(\mathcal{C}) \longrightarrow \mathcal{C}(\F_q)$ uniquely defined by the condition that each divisor $D$ of degree $d$ is linearly equivalent to the divisor $ \sigma(D) + (d-1)P_{\infty}$. This map $ \sigma $ is a group morphism, it passes to linear equivalence, and induces an isomorphism of the degree 0 class group ${Cl}^{0}(\mathcal{C})$. First, let us recall the result obtained by Randriambololona in~\cite[Proposition 4.3]{Randriam}.

\begin{Proposition}\label{Randriambololona}
Let $\mathcal{C}$ be an elliptic curve over $ \F_{q}$, $n$ be an integer. Suppose that $\mathcal{C}$ admits a closed point 
$Q$ of degree $n$. Let $G$ be an effective divisor on $\mathcal{C}$, and write
         $$ G = u_{1}P_{1}+ \cdots + u_{N}P_{N} $$
where $P_{i}$ are pairwise distinct closed points, of degree $\deg P_{i}$, so 
          $$\deg G = \sum_{i=1}^{N}\deg P_{i}.u_{i}.$$        
\noindent
Then 
     \begin{equation}\label{eqR}
          \mu_{q,\mathcal{C}}(n,1) \leq \sum_{i=1}^{N} \mu_{q}(\deg P_{i},u_{i}),     
     \end{equation}

 provided if one of the following conditions is satisfied: 
      \begin{enumerate}
         \item $\mathcal{C}$ admits at least three points of degree one and $\deg G \geq 2n$.
         \item  $\mathcal{C}$ admits two points of degree one and $\deg G \geq 2n$, 
         with $\sigma(G) \neq P_{\infty}$.                  
         \item $\mathcal{C}$ admits only two points of degree one and $ \deg G \geq 2n + 1 $.
         \item $\mathcal{C}$ admits only one point of degree one and $ \deg G \geq 2n + 3 $.
      \end{enumerate}                                      

\end{Proposition}

\noindent
The above result gives sufficient conditions to construct an elliptic bilinear algorithm of type Chudnovsky (cf. Theorem \ref{AlgoRandriam}) 
from an elliptic  curve $\mathcal{C}$ and from an effective divisor $ G $  on $\mathcal{C}$. However, note that unlike the case of genus $0$ in \cite[Proposition 4.2]{Randriam}, it does not give sufficient conditions to construct a symmetric  elliptic bilinear algorithm because of Assertion (1). 
We propose an improvement of this result in two points:

\begin{itemize}
 \item firstly, we give on the one hand sufficient conditions to construct symmetric algorithms, and on the other hand, we give explicit equations of elliptic curves which are more convenient that the above conditions.
 \item Finally, our new result allows us to bound $\mu_{q,\mathcal{C}}(n,1)$, not only with the best known bounds for $\mu_{q}(\deg P_{i},u_{i})$ as Theorem~\ref{Randriambololona} suggests, but also with bounds for $\mu_{q,\mathcal{C}}(\deg P_{i},u_{i})$ derivated from the same elliptic curve $\mathcal{C}$.
\end{itemize}

In order to achieve this, we need to know the number of 2-torsion points on divisor class group as signaled by Cascudo, 
Cramer and Xing in \cite{cacrxi} (Cf. also \cite[Chapter 9]{casc}) relatively to the proof of Claim in \cite[Theorem 3.1]{shtsvl}. 
In particular, we need to know it when 
the number of rational points of the elliptic curve $E/\F_q$ defined over a finite field $\F_q$ of odd characteristic is equal 
to four. Note that if $E$ denotes an elliptic curve defined over a field $K$, and $\ell~\in \Z$ such that $\ell$ 
is prime with the characteristic of $K$, then the group of $\ell$-torsion points $E[\ell]$ is isomorphic to 
$\frac{\Z}{\ell\Z}\times \frac{\Z}{\ell\Z}$ but it holds on the algebraic closure. Hence, the only mean
to precisely know the subgroup of $\ell$-torsion points $E_q[\ell]$ over $K$ (and not only an upper bound) 
is to know the structure of $E(K)$.\\

\begin{Lemma}\label{structellip}
Let $q$ be a prime power with odd characteristic and let $E/\F_q$ be an elliptic curve defined over $\F_q$. Then the group $E(\F_q)$ of the $\F_q$-rational points of $E/\F_q$ is isomorphic to the finite abelian group $G=\frac{\Z}{2\Z}\times \frac{\Z}{2\Z}$ in the following cases: 

\begin{enumerate}
\item $q=3$ and $E/\F_q$ admits the following equation up to isomorphism: \[y^2 + y + 2x^3 + x + 1 = 0.\]
\item $q=5$ and $E/\F_q$ admits the following equation up to isomorphism: \[ y^2 + 4x^3 + 4x = 0.\]
\item $q=7$ and $E/\F_q$ admits the following equation up to isomorphism: \[ y^2 + 6x^3 + 1 = 0.\] 
\item $q=9$ and $E/\F_q$ admits the following equation up to isomorphism: \[ y^2 + (x + 1)y + 2x^3 + x^2 + ax + 1 = 0,~ \hbox{ where } \F_9=\F_3(a).\]
\end{enumerate}
\end{Lemma}

\begin{proof}
Let $N_1(E(\F_q))$ denote the number of $\F_q$-rational points of the elliptic curve $E$ defined over $\F_q$.
It is known that the number of $\F_q$-rational points of  $E/\F_q$ is equal to $N_1(E(\F_q))=q+1-m$ 
where the integer $m$ is the trace of the Frobenius which satisfies $\mid m \mid\leq2\sqrt{q}$. Hence, if $N_1(E(\F_q))=4$,
the only concerned finite fields are $\F_3$, $\F_5$, $\F_7$, and $\F_9$. In this case, we have: $m=0$ if $q=3$, 
$m=2$ if $q=5$, $m=4$ if $q=7$, and $m=6$ if $q=9$. Then, by Theorem 2.1 in~\cite{vlad} (cf. also \cite[Theorem 2.4.31]{ShpaVla}), 
if $q\neq 3$, $E(\F_q)$ is isomorphic to $\frac{\Z}{2\Z}\times \frac{\Z}{2\Z}$, else $E(\F_q)$ is either cyclic or isomorphic 
to $\frac{\Z}{2\Z}\times \frac{\Z}{2\Z}$. It is easy to chek that each above curves defined over the corresponding finite field has 
four  rational points. Moreover, the curve $y^2 + y + 2x^3 + x + 1 = 0$ defined over $\F_3$ is not cyclic. 
\end{proof}

\noindent 
We start by stating the first of our three main results, namely Proposition~\ref{BBT} which is an improvement of Proposition~\ref{Randriambololona}. Then, we prove in Theorem~\ref{n2q} that our new Proposition~\ref{BBT} allows us to construct, asymptotically with respect to the integer $n$, multiplication algorithms with symmetric bilinear complexity in $O(n(2q)^{\log_q^*(n)})$. Finally, Theorem~\ref{complexity} shows that the complexity of the construction of such algorithms is in time polynomial in $n$.

\begin{Proposition}\label{BBT}
Let $q$ be a prime power and let $\mathcal{C}$ be an elliptic curve defined over $ \F_{q}$. 
Then, for  any integer $n$ such that $n\geq 7$ if $q=2$, $n\geq 4$ if $q=3$ and $n\geq 3$ if $q\geq 4$, there exists a symmetric elliptic bilinear algorithm of type Theorem~\ref{AlgoRandriam} constructed from the curve $\mathcal{C}$ 
and from an effective divisor $$ G = u_{1}P_{1}+ \cdots + u_{N}P_{N} $$  
on $\mathcal{C}$ such that 
$$ \mu^{sym}_{q,\mathcal{C}}(n,1) \leq \sum_{i=1}^{N} \mu^{sym}_{q,\mathcal{C}}(\deg P_{i},u_{i}),$$ 
\noindent
where the $P_i$ are $N$ pairwise distinct  closed points, of degree  $\deg P_i=d_i$, and the $u_i$ 
are strictly positive integers, provided one of the following conditions is satisfied: 

a) the curve  $\mathcal{C}$ admits one of the following equations up to isomorphism: 
\[
 \begin{array}{l}
    y^2+y+(x^3+x+1)=0,~  \hbox{ if }~ q=2,\\
    y^2-(x^3+2x+2)=0,~   \hbox{ if }~ q=3,\\
    y^2+y+(x^3+a)=0,~   \hbox{ if }~ q=4~\hbox{and}~\F_4=\F_2(a),\\
 \end{array}
\]
and
$$\sum_{i=1}^{N}u_id_i\geq 2n+3.$$


b) The curve  $\mathcal{C}$ admits one of the following equations up to isomorphism: 
\[
\begin{array}{l}
 y^2+xy+x^3+x^2+1=0 ~ \hbox{ if } ~q=2,\\
 y^2-(x^3+2x^2+2)=0 ~  \hbox{ if }~ q=3,\\
 y^2+xy+(x^3+ax^2+1)=0 ~ \hbox{ if }~ q=4~\hbox{and}~ \F_4=\F_2(a),\\
 y^2-(x^3+2x)=0 ~        \hbox{ if }~ q=5,\\
\end{array}
\]
and either
$$\sum_{i=1}^{N}u_id_i\geq 2n+1$$ or  $$\sum_{i=1}^{N}u_id_i=2n~ \hbox{ with }~ \sigma(G) \neq P_{\infty}.$$

c) The curve  $\mathcal{C}$ admits one of the following equations up to isomorphism: 
\[
\begin{array}{l}
  y^2 + y + 2x^3 + x + 1 = 0~ \hbox{ if }~ q=3,\\
  y^2 + 4x^3 + 4x = 0~ \hbox{ if }~ q=5,\\
  y^2 + 6x^3 + 1 = 0~ \hbox{ if }~ q=7,\\
  y^2 + (x + 1)y + 2x^3 + x^2 + ax + 1 = 0~\hbox{ if }~ q=9 ~\hbox{and}~ \F_9=\F_3(a),\\
\end{array}
\] 
 and $$\sum_{i=1}^{N}u_id_i\geq 2n+1.$$

d) The equation of the curve  $\mathcal{C}$ is different from the above cases up to isomorphism and $$\sum_{i=1}^{N}u_id_i\geq 2n.$$ 

Particularly for $q = 2$, elliptic curves  are
\[
\begin{array}{l}
 y^2 + y + x^3 = 0 \\
 y^2 + y + x^3 + x = 0 \\ 
 y^2 + xy + x^3 + 1 = 0,
\end{array}
\] 

and for $ q= 3$, we obtain
\[
 \begin{array}{l}
 y^2 + 2x^3 + 2x = 0 \\
 y^2 + 2x^3 + x + 2 = 0 \\
 y^2 + 2x^3 + 2x^2 + 2 = 0\\ 
 y^2 + 2x^3 + 2x^2 + 1 = 0\\
 y^2 + 2x^3 + x^2 + 2 = 0. 
 \end{array}
\]
\end{Proposition}

\noindent
\begin{proof} To begin with, it is necessary to prove that for any $q$ and any integer $n\geq 3$, there exists a closed point $Q$ 
of degree $n$ and a divisor $D$ such that both evaluation maps, in a symmetric algorithm of type Theorem~\ref{AlgoRandriam}, 
are surjective and injective. For the sake of simplicity, we use the language of algebraic function fields. It is well known, by Lemma 2.1 
of~\cite{chau}, that for any integer $n\geq 3$ and for any prime power $q\geq 4$, all the elliptic function fields defined over $\F_q$  
have at least a place $Q$ of degree $n$. The cases $q=3$ and $q=2$ had still to be dealt with. 
For $q=3$ and $n\geq 4$ and for $q=2$ and $n\geq 7$, we have: $n\geq 2log_q(\frac{3q^{1/2}}{q^{1/2}-1})$, which proves
that by \cite[Corollary V.2.10 (c)]{Stich}, there exists at least a place Q of degree respectively  $n\geq7$ for $q=2$ and $n\geq4$
 for $q=3$ for any elliptic algebraic function field $E/\F_q$ defined over $\F_q$.
Let us prove now that for any $q$ and for any elliptic function field $E/\F_q$  defined over $\F_q$, there exists a divisor $D$ such that we can construct a symmetric algorithm of type Theorem~\ref{AlgoRandriam}. 
\begin{itemize}
 \item Proof of a). If $N_1(\mathcal{C}(\F_q))=1$, we know by~\cite{macr} and~\cite{maqu}, that the elliptic solutions to the divisor class number one problem are given by the equations  of the case a). Moreover, as $\deg G\geq 2n+3$, it is sufficient to take a divisor $D=D_1=D_2$ of degree $n+1$ and conditions (i') and (ii') of Theorem~\ref{AlgoRandriam} are trivially satisfied because respectively $\deg(2D-G)<0$ and $\deg (D-Q)>2g-2$ with $g=1$.
 \item Proof of b). If $N_1(\mathcal{C}(\F_q))=2$, we know  by~\cite{Brigand1} and~\cite{Brigand2} that the elliptic solutions to the divisor class number two problem are given by the equations  of the case b). Moreover, it does mean that there exists a divisor ${\mathcal R}$ of degree zero which is not linearly equivalent to the divisor zero. Then by taking $D=D_1=D_2= {\mathcal R}+Q$, the condition (ii') of Theorem~\ref{AlgoRandriam} is satisfied. Moreover, it also means that the Jacobian of $\mathcal{C}/\F_q$ is of 2-torsion, so $\sigma(2D-G)=\sigma(2D)+\sigma(-G)=\sigma(G)$. Then if  $\sigma(G)\neq P_{\infty}$ and $\deg G=2n$, then $2D-G$ is not linearly equivalent to the divisor zero which proves that $2D-G$ is non special of degree zero and the condition (i') of Theorem~\ref{AlgoRandriam} is satisfied. Else,  $\deg G\geq2n+1$ and thus $\deg (2D-G)<0$ and $2D-G$ is trivially of dimension zero which implies the condition (i') and proves the case b).
\item Proof of c) and d). If $N_1(\mathcal{C}(\F_q))\geq 3$, it is sufficient to prove the following inequality by \cite[Chapter 9]{casc} relatively to 
the proof of Claim in \cite[Theorem 3.1]{shtsvl}(cf. also \cite{cacrxi}):
$$\mathcal{C}(\F_q)[2]+1<N_1(\mathcal{C}(\F_q)$$ where $\mathcal{C}(\F_q)[2]$ denotes the number of 2-torsion rational points 
of the elliptic curve $\mathcal{C}/\F_q$.  It is known that the number of 2-torsion points of an elliptic curve defined over a finite field $\F_q$ 
is at most four (cf.~\cite{wate} and~\cite{vlad}). Hence the inequality is satisfied for any elliptic curve having at least six rational points. 
We discuss thereafter, particular cases where $ 3 \leq N_1(\mathcal{C}(\F_q) \leq 5$.
 \begin{itemize}
\item If  $N_1(\mathcal{C}(\F_q))=3 \hbox{ or } 5$, there is no nontrivial 2-torsion point and so $\mathcal{C}(\F_q))[2]=1$ and 
the inequality is also satisfied.
\item If $N_1(E/\F_q))=4$ and
\begin{enumerate}
 \item if the characteristic of $\F_q$ is even, then by a general theorem of Weil (cf.~\cite[Theorem 11.12]{rose}) applied to elliptic abelian varieties, the number of 2-torsion points of an elliptic curve defined over a finite field $\F_q$ is at most two and the inequality is also satisfied.
 \item if the characteristic of $\F_q$ is odd, only all elliptic curves in Lemma~\ref{structellip} admit four 2-torsion points and the inequality is not satisfied. However, there exists a divisor ${\mathcal R}$ of degree zero which is not linearly equivalent to divisor zero. Consequently, by taking $D=D_1=D_2= {\mathcal R}+Q$, the condition (ii') of Theorem~\ref{AlgoRandriam} is satisfied as well as the condition (i') since $\deg (2D-G)<0$, which gives c) and the proof is complete.
\end{enumerate}
\end{itemize}
\end{itemize}
\end{proof}

\begin{Definition}
The iterated logarithm of $n$, written $ \log^*_q(n)$  defined by the following recursive function:

$$
\log^*_q(n) = \left\{
    \begin{array}{ll}
        0 & \mbox{if } n \leq 1 \\
        1 + \log^*_q(\log_q(n)) & \mbox{otherwise,}
    \end{array}
\right.
$$

\noindent
corresponds to the number of times the logarithm function must be iteratively applied to $n$ before the result is less than or equal to 1. 
\end{Definition}

\begin{Theorem}
Let $q$ be a prime power and let $\mathcal{C}$ be an elliptic curve defined over $ \F_{q}$. 
Then, for  any integer $n$ such that $n\geq 7$ if $q=2$, $n\geq 4$ if $q=3$ and $n\geq 3$ if $q\geq 4$, there exists a symmetric elliptic bilinear algorithm of type Theorem~\ref{AlgoRandriam} constructed from the curve $\mathcal{C}$ such that
$$ \mu^{sym}_{q,\mathcal{C}}(n)~\in~O\left(n(2q)^{\log_q^*(n)} \right). $$
\noindent
Notice that $(2q)^{\log_q^*(n)}$ is a very slowly growing function, as illustrated in Table~\ref{IL}.

\begin{table}[h]
\begin{center}
\caption{Values for $(2q)^{\log_q^*(n)}$ for $q=2$ and $ n \leq 2^{65536}$.}\label{IL} \vspace{0,1cm}
~~\\
\begin{tabular}{|c|c|c|}
\hline
$n$ & $\log^*(n)$ & $(2q)^{\log_q^*(n)}$\\
\hline
$(1,2]$ & 1  & 4\\
\hline
$(2,4]$ & 2 & 16\\
\hline
$(4,16]$ & 3 & 64 \\
\hline
$(16,65536]$ & 4 & 256\\
\hline
$(65536, 2^{65536}]$ & 5 & 1024\\
\hline
\end{tabular}
\end{center}
\end{table}

\label{n2q} 
\end{Theorem}

\begin{proof}
Without loss of generality, let $C$ be an elliptic curve which the model does not appear in case a) and b) of  Theorem~\ref{BBT}. Let $G$ be the divisor on $C$ such that 
$$ G = u_{1}P_{1}+ \cdots + u_{N}P_{N}.$$

\noindent 
Concentrating on the worst case, we can assume that
\begin{itemize}
 \item we do not use derivative evaluation, that is $ u_i = 1 ~~\mbox{for}~~ 1\leq i \leq N $,
 \item we only use places of a fixed degree, that is $\deg(P_1)= \ldots = \deg(P_N)=d_1$.
\end{itemize}
With these assumptions $ G = P_1 + \cdots + P_{B_{d_1}} $, where $B_{d_1}$ denotes the number of places of degree $d_1$. From Theorem~\ref{BBT}, if $C$ is one of the elliptic curve of case d) and $\deg(G)=d_1B_{d_1}\geq 2n$, then
\begin{equation}\label{Randriam2}
 \mu^{sym}_{q,\mathcal{C}}(n) \leq \sum_{i=1}^{B_{d_1}}\mu^{sym}_{q,\mathcal{C}}(\deg P_i)= B_{d_1}\mu^{sym}_{q,\mathcal{C}}(d_1).
\end{equation}
\noindent
From~\cite[Corollary 5.2.10]{Stich} applied to elliptic curves, we know that  $B_{d_1}$ verifies 
$$ \frac{q^{d_1}}{{d_1}}-9\frac{q^{{d_1}/2}}{{d_1}} < B_{d_1} < \frac{q^{d_1}}{{d_1}}+9\frac{q^{{d_1}/2}}{{d_1}}.$$ 
\noindent
Asymptotically, $B_{d_1}~\in~O\left( \frac{q^{d_1}}{d_1}\right)$ and then $\deg(G)~\in~O(q^{d_1})$.
\noindent
Let ${d_1}$ be the smallest integer such that $q^{d_1} \geq 2n  $, then $ q^{{d_1}-1} < 2n $ and we have $d_1 ~\in~O \left(\log_q(2n) \right)$.
\noindent
Thus
$$ \mu^{sym}_{q,\mathcal{C}}(n) ~\in~O \left( B_{d_1} \mu^{sym}_{q,\mathcal{C}}(d_1) \right)$$
and then
$$ \mu^{sym}_{q,\mathcal{C}}(n) ~\in~ O \left( \frac{2nq}{d_1} \mu^{sym}_{q,\mathcal{C}}(d_1) \right).$$
\noindent
Using recursively the process, we obtain
$$ \mu^{sym}_{q,\mathcal{C}}({d_1}) ~\in~O \left( \frac{2{d_1}q}{d_2}\mu^{sym}_{q,\mathcal{C}}(d_2) \right),$$
\noindent 
where $d_2~\in~O \left(\log_q(2d_1)\right)$.
\noindent
With this procedure, we have
$$ \mu^{sym}_{q,\mathcal{C}}(n) ~\in~O \left( \frac{2nq}{{d_1}}\cdot\frac{2d_1q}{d_2}\cdot~\cdots~\cdot\frac{2d_{k-2}q}{d_{k-1}}\cdot2d_{k-1}q\frac{\mu^{sym}_{q,\mathcal{C}}(d_k)}{d_k} \right), $$
\noindent
with $d_i ~\in~ O \left(\log_q(2d_{i-1}) \right)$, for $1 \leq i \leq k$, and consequently
$$ \mu^{sym}_{q,\mathcal{C}}(n) ~\in~O \left( n(2q)^{k} \cdot \frac{\mu^{sym}_{q,\mathcal{C}}(d_k)}{d_k} \right).$$
\noindent
Let $k=\log_q^*(2n)$, then we have
$$ d_k ~\in~ O \left(\underbrace{\log_q(\log_q(\ldots(\log_q}_{k~terms}(2n))\ldots))\right) \leq 1,$$
\noindent
and thus 
$$\frac{\mu^{sym}_{q,\mathcal{C}}(d_k)}{d_k} \leq 1. $$
\noindent
Finally $$ \mu^{sym}_{q,\mathcal{C}}(n)~\in~ O\left( n \cdot (2q)^{\log^*_q(n)}\right).$$
\end{proof}

\begin{Corollary}\label{coro1}
For  any integer $n$ such that $n\geq 7$ if $q=2$, $n\geq 4$ if $q=3$ and $n\geq 3$ if $q\geq 4$, there exists a symmetric elliptic bilinear algorithm of type Theorem~\ref{AlgoRandriam} constructed from a curve of genus one and from an effective divisor $$ G = u_{1}P_{1}+ \cdots + u_{N}P_{N}, $$
on this curve, where the $P_i$ are $N$ pairwise distinct  closed points, of degree  $\deg P_i=d_i$, and the $u_i$ 
are strictly positive integers, such that 
 $$ \mu^{sym}_{q,1}(n)~\in~ O\left( n \cdot (2q)^{\log^*_q(n)}\right).$$
\end{Corollary}

\begin{proof}
The proof is similar to the proof of Theorem~\ref{n2q}. Indeed, for a fixed genus $g$, the number $B_d$ of places of degree $d$, as claimed in~\cite[Corollary 5.2.10]{Stich}, is such that 
  $$ \frac{q^{d}}{{d}}-(2+7g)\frac{q^{{d}/2}}{{d}} < B_{d} < \frac{q^{d}}{{d}}+(2+7g)\frac{q^{{d}/2}}{{d}}.$$
Thus, for each curve of genus $g$, $B_d$ is asymptotically the same. Consequently, changing the model of the elliptic curve does not change the proof, and does not change asymptotically the bilinear complexity. 
\end{proof}

\noindent
Elliptic curves have already been used to bound the bilinear complexity of multiplication (see for example the work of Shokrollahi~\cite{Shokrollahi}, Ballet~\cite{Ballet6}, and Chaumine~\cite{chau}). Recently, Couveignes and Lercier~\cite{CouvLer} proposed a multiplication algorithm for finite field extensions $\F_{q^n}$, using normal elliptic bases. Their multiplication tensor  consists in 5 convolution products, 2 component-wise products, 1 addition and 3 subtractions. Note that convolution products can be computed at the expense of $O\left(n\log n \log |\log(n)|\right)$ operations in $\F_q$. Asymptotically, the tensor they produce is not competitive with ours from the point of view of bilinear complexity.
 
\subsection{Complexity of the construction}
Studies on bilinear complexity are well advanced, however we do not know a single polynomial construction of bilinear multiplication algorithm with linear or quasi-linear multiplicative complexity. 
In the case of bilinear multiplication algorithm with linear multiplicative complexity, namely the case of the usual strategy based upon the construction with growing genus, we cannot give information about the complexity of construction. Indeed, it is completely unclear how to construct explicitely points of high degree~\cite[Section 4, Remarks 5]{shtsvl}. However, using the new strategy with elliptic curves, we show that we can polynomially construct symmetric elliptic bilinear multiplication algorithms with quasi-linear multiplicative complexity.

\begin{Lemma}\label{PlaceDegreeN}
Let $E$ be an elliptic curve defined over $\F_q$ and let $F/\F_q$ be the associated elliptic function field. Then we can construct a degree $n$ place of $F/\F_q$ in time polynomial in $n$.
\end{Lemma}

\begin{proof}
In order to construct a degree $n$ place $Q$ of the elliptic function field $F/\F_q$, firstly we have to construct a rational point $\mathscr{P}= (x_{\mathscr{P}},y_{\mathscr{P}})$ of $E$ defined over $\F_{q^n}$ and then, we need to apply to the point $\mathscr{P}$, $n$-times the Frobenius map $\varphi$ defined by
\[ 
\begin{array}{llll}
\varphi: & E(\overline{\F}_{q^n}) & \longrightarrow & E(\overline{\F}_{q^n})\\
         &        ( x, y )       &  \longmapsto    &     (x^q, y^q).          \\ 
\end{array}
\]
Thus the orbit of $\mathscr{P}$ obtained under the action of $\varphi$ is a degree $n$ place. In 2006, Shallue and Van De Woestijne~\cite{Shallue} gave a deterministic polynomial-time algorithm that computes a nontrivial rational point given a Weierstrass equation for the elliptic curve. More precisely, they performed the computation of a nontrivial rational point on an elliptic curve $E$ defined over $\F_q$ in time polynomial in $\log(q)$. It follows that $\mathscr{P}$  can be constructed in time polynomial in $\log(q^n)$, and thus in time polynomial in $n$ since $q$ is fixed. The action of the Frobenius map $\varphi$ on the point $\mathscr{P}$ is simply a modular exponentiation that can be done polynomially. Consenquently, constructing a degree $n$ place of an elliptic function field can be done in time polynomial in $n$.
\end{proof}

\begin{Theorem}\label{complexity}
Given an elliptic curve $E$ defined over $\F_q$, one can polynomially construct a sequence $ \mathscr{A}_{q,n} $ of symmetric elliptic bilinear multiplication algorithms in finite fields  $\F_{q^n}$ for the given sequence $n \rightarrow +\infty$ such that 
    \[ \mu^{sym}_{q,E}(\mathscr{A}_{q,n})~\in~O\left( n(2q/K)^{\log_q^{*}(n)}\right),\]
where $K=2/3$ if the characteristic of $\F_q$ is 2 or 3, and $K=5/8$ otherwise.
\end{Theorem}

\begin{proof}
Let $F/\F_q$ be the elliptic function field associated to the curve $E$. According to the proof of Theorem~\ref{n2q}, to construct a symmetric elliptic multiplication algorithm in $ \F_{q^n} $ over $\F_q$, we first have to construct places and divisor of certain degree. Indeed, we need to construct 
 \begin{itemize}
  \item a place $Q$ of degree $n$ of $F/\F_q$,
  \item a divisor $D$ of degree $n$ of $F/\F_q$,
  \item a sufficient number $N$ of degree $d$ places of $F/\F_q$, such that the degree of the divisor $G$ formed by these $N$ places, is greater or equal to $2n$. 
 \end{itemize}
\noindent 
The divisor $D$ and the place $Q$ are equivalent in terms of construction~(in practice we can take any place to construct a divisor~\cite{Ballet3}), so their complexities of construction are similar and from Lemma~\ref{PlaceDegreeN} this complexity is in time polynomial in $n$. The point now is to construct sufficiently places of degree $d$ of $F/\F_q$. To achieve this, from lemma~\ref{PlaceDegreeN} it suffices to construct rational points of the curve $E$ over $\F_{q^d}$. Icart~\cite{Icart} shows that it is possible to construct deterministically, a constant proportion $K$ of the number of rational points of an elliptic curve defined over $\F_{q}$. More precisely, his method allows us to construct $K = 5/8$ of the number of rational points in time polynomial in $\log^3(q)$. Note that if the characteristic of $\F_q$ is 2 or 3, Farashi et al~\cite{FaShpaVo} proved that $K =2/3$. This implies that asymptotically, we can construct in time polynomial in $\log^3(q^d)$, a sufficient number of places of degree $d$ of $F/\F_q$ by choosing $d$ such that $q^d \geq 2n/K$. Finally, the complexity of construction of places of degree $d$ is polynomial in $\log^3(n)$, thus polynomial in $n$. In conclusion, we can polynomially construct symmetric elliptic billinear multiplication algorithms since for a given divisor $D$, construct vector spaces $\mathscr{L}(D),~\mathscr{L}(2D)$, associated basis $\mathscr{B}_D,~\mathscr{B}_{2D}$ and evaluation maps $Ev_{Q},~Ev_{P}$ can be done polynomially~\cite[Section 4, Remarks]{shtsvl}
(cf. also ~\cite[p. 509, Remark 4.3.33]{ShpaVla}).
\end{proof}

\begin{Remark}
This complexity can indeed be refined. We plan to study in detail this problem in a forthcoming work.  
\end{Remark}

\section{Upper Bounds and Example of construction}
\noindent
Using our strategy, we propose in this section:
\begin{itemize}
 \item upper bounds of symmetric bilinear complexity for large extension of finite fields $\F_2$ and $\F_3$, and
 \item  an example of a multiplication algorithm construction.
\end{itemize}
In order to obtain the best bounds of symmetric bilinear complexity, we use our Theorem~\ref{BBT} not with bounds $\mu^{sym}_{q,\mathcal{C}}(\deg P_{i},u_{i})$ derivated from the same elliptic curve $\mathcal{C}$, but with the better known bounds for $\mu_{q}(\deg P_{i},u_{i})$ as in Theorem~\ref{Randriambololona}. Moreover, for a fixed $n$, to obtain the best bounds of symmetric bilinear complexity, we need to find the best curve of genus one and thus we compute, not $\mu^{sym}_{q,\mathcal{C}}(n)$ but  $\mu^{sym}_{q,1}(n)$. We note throughout the rest of the paper $\mu^{sym}_{q}(n)$ instead of $\mu^{sym}_{q,1}(n)$.

\subsection{New Bounds}

In elliptic curve cryptography, the NIST suggests to use finite fields with  $2^{163}, 2^{233}$, $2^{283}, 2^{409}$ and $2^{571}$ elements~\cite{NIST}. 
Randriambololona in~\cite{Randriam} obtained the following bound $$ \mu^{sym}_{2}(163) \leq 910.$$
We improve this bound  $$ \mu^{sym}_{2}(163) \leq 906. $$\\
\noindent
In order to upgrade $\mu^{sym}_{2}(163)$, we seek out of the curves given in Theorem \ref{BBT}, the one which provides the lowest bilinear complexity. Using only higher multiplicity  with degree one and degree two places, the best curve turns out to be $y^2 + y + x^3 = 0$. This curve has 3 points of degree 1, and the lowest bilinear complexity is obtained with the divisor $G$ of degree $2.163$ defined as follows:\\
we take all 3 points of degree 1 with multiplicity 4, all 3 points of degree 2 with multiplicity 2, and all 2 points of degree 3, all 6 points of degree 5, all 11 points of degree 6 and 25 points of degree 8, all with multiplicity 1. Then the degree of G is
$$ \deg G = 3.1.4 + 3.2.2 + 2.3.1 + 6.5.1 + 11.6.1 + 25.8.1 = 326 = 2.163. $$

\noindent 
From Theorem \ref{BBT} used with the best known bounds for $\mu_{q}(\deg P_{i},u_{i})$ and values of Table \ref{T3} we obtain\\
 
\begin{tabular}{cl}
 $ \mu^{sym}_{2}(163) $ & $ \leq 3.\mu_{2}(1,4) + 3.\mu_{2}(2,2) + 2.\mu_{2}(3,1) + 6.\mu_{2}(5,1) + 11.\mu_{2}(6,1) + 25.\mu_{2}(8,1) $ \\
                  & $ \leq 3.\widehat{M_2}(4) + 3.\mu_{2}(2)\widehat{M_2}(2) + 2.\mu_{2}(3) + 6.\mu_{2}(5) + 11.\mu_{2}(6) + 25.\mu_{2}(8) $ \\
                  & $ \leq 3.8 + 3.3.3 + 2.6 + 6.13 + 11.15 + 25.24 $ \\
 $ \mu^{sym}_{2}(163) $ & $ \leq 906.$\\
\end{tabular}
\vspace{0,1cm}

\noindent
Table \ref{Tablemu2n} (respectively Table \ref{Tablemu3n}) represents optimal bounds for $\mu^{sym}_{2}(n)$ (respectively $\mu^{sym}_{3}(n)$) and the size of extension for $\F_2$ is in accordance with the NIST for elliptic curve cryptography. The column $N$ represents the number of places of arbitrary degrees used to obtain the optimal bound, and column $U$, the associated order for derivative evaluation. As example, for $n=233$, we obtain the lower bound $1340$ using the elliptic curve (up to isomorphism) defined by $y^2 + xy = x^3 + 1$. This lower bound is achieved with $N=[ 4, 2, 0, 2, 8, 8, 10, 34 ]$ and  $U = [ 5, 2, 1, 1, 1, 1, 1, 1 ]$, meaning that we use 4 degree one places with multiplicity $u_1$ equals 5, 2 degree two places with multiplicity $u_2$ equals 2 and the remainder with multiplicity 1.

\begin{table}[h]
\caption{Optimal bounds for $\mu^{sym}_{2}(n)$. }
~~\\
 \begin{tabular}{|c|c|c|l|l|}
\hline
  $n$ & $\mu^{sym}_2(n)$ & Elliptic Curve  & \hspace{2,5cm}N  & \hspace{1,5cm} U  \\

\hline
  163 &   906      &   $y^2 + y + x^3=0$   &  $[ 3, 3, 2, 0, 6, 11, 0, 25 ]$ &  $[ 4, 2, 1, 1, 1, 1, 1, 1 ]$        \\
\hline
  233 &   1340      &  $y^2 + xy + x^3 + 1=0$  &  $[ 4, 2, 0, 2, 8, 8, 10, 34 ]$ & $[ 5, 2, 1, 1, 1, 1, 1, 1 ]$          \\
\hline
  283 &   1668      &  $y^2 + xy + x^3 + 1=0$  &  $[ 4, 2, 0, 2, 8, 8, 14, 34, 8 ]$ & $[ 5, 2, 1, 1, 1, 1, 1, 1, 1 ]$          \\
\hline
  409 &  2495       &  $y^2 + xy + x^3 + 1=0$ &  $[ 4, 2, 0, 2, 8, 8, 16, 34, 0, 31 ]$ & $[ 5, 2, 1, 1, 1, 1, 1, 1, 1, 1 ]$        \\
\hline
  571 &  3566       &   $y^2 + xy + x^3 + 1=0$  &  $[ 4, 2, 0, 2, 8, 8, 16, 34, 2, 62 ]$  & $[ 5, 1, 1, 1, 1, 1, 1, 1, 1, 1 ]$       \\
\hline

\end{tabular}\label{Tablemu2n}
\end{table}

\begin{table}[h]
\caption{Optimal bounds for $\mu^{sym}_{3}(n)$. }
~~\\
 \begin{tabular}{|c|c|c|l|l|}
\hline
  $n$ & $\mu^{sym}_3(n)$ & Elliptic Curve  & \hspace{1,5cm}N  & \hspace{1cm} U  \\

\hline
  57 &   234      &   $y^2 + 2x^3 + 2x^2 + 1=0$   &  $[ 3, 6, 11, 15]$ &  $[ 3, 1, 1, 1]$        \\
\hline
  97 &   426      &  $y^2 + 2x^3 + 2x^2 + 1=0$  &  $[ 3, 6, 11, 15, 16 ]$ & $[ 3, 1, 1, 1, 1]$          \\
\hline
  150 &   681      &  $y^2 + 2x^3 + 2x^2 + 1=0$  &  $[ 3, 6, 11, 14, 38 ]$ & $[ 3, 1, 1, 1, 1]$          \\
\hline
  200 &  925       &  $y^2 + 2x^3 + x^2 + 1=0$ &  $[ 2, 5, 12, 21, 47, 5 ]$ & $[ 3, 1, 1, 1, 1, 1]$        \\
\hline
  400    &  1926       & $y^2 + 2x^3 + x^2 + 1=0$ & $[ 2, 5, 12, 21, 47, 72 ]$ & $[ 2, 1, 1, 1, 1, 1 ]$  \\
\hline
\end{tabular}\label{Tablemu3n}
\end{table}

\subsection{Effective multiplication algorithm in $ \F_{3^{57}}$}

In this section, we choose to present the construction of the multiplication algorithm in $ \F_{3^{57}}$ with $234$ bilinear multiplications, using elliptic curves, points of higher degree and higher multiplicity.

\subsubsection{\textbf{Method}}
Let $ \alpha $ and $ \beta $  be two elements of $ \F_{3^{57}} $. Since there exists a point $Q$ of degree 57, the residue class field $ \mathcal{O}_{Q}/Q $ is isomorphic to $ \F_{3^{57}} $ and we can consider that both elements are in $ \mathcal{O}_{Q}/Q $. Furthermore, there exists a divisor $D$ such that the evaluation map
  $$ Ev_{Q} : \mathscr{L}(D) \longrightarrow \frac{\mathcal{O}_{Q}}{Q} $$
  $$ \hspace{2cm}f ~~~~\longmapsto f(Q) $$
is surjective. Hence there exist two functions $f_{\alpha},~f_{\beta}~\in~\mathscr{L}(D)$ such that $ Ev_{Q}(f_{\alpha})= \alpha, ~\mbox{and}~  Ev_{Q}(f_{\beta})= \beta $.
Finally, to obtain the product $ \alpha.\beta $, we compute 
$ Ev_{Q}(f_{\alpha}.f_{\beta})= \alpha.\beta. $
At this step, we have to construct the only $f_{\gamma}~\in~\mathscr{L}(2D)$ such that $~f_{\alpha}.f_{\beta} = f_{\gamma}$. The unicity of $f_{\gamma}$ comes from the injectivity of the second evaluation map $ Ev_{\mathcal{P}}$.
Consider  $f_{\alpha}=\sum_{i=1}^{57}a_{i}f_{i},~f_{\beta}=\sum_{i=1}^{57}b_{i}f_{i}$  and
let $ f_{\gamma} $ be the product of $f_{\alpha}$ and $f_{\beta}$ given by the relation
\begin{equation}\label{Relation} 
\begin{array}{cc} 
    \underbrace{\left( \sum_{i=1}^{57}a_{i}f_{i}\right).\left( \sum_{i=1}^{57}b_{i}f_{i}\right)}= & \underbrace{\sum_{i=1}^{114}c_{i}f_{i}},\\
 M  & C
 \end{array} 
\end{equation} 
\noindent
where M and C are the matrix representation of the relation (\ref{Relation}).

\subsubsection{\textbf{Choice of the degree of places}}

For a fixed $n$, it is not clear how to find the maximal degree of places to use, but in elliptic case it is easy to perform it. 
From the proof of Theorem \ref{n2q}, the maximal degree $d$ of places must verify $ q^d > 2n $, so $d$ equals 5 for $n=57$.

\subsubsection{\textbf{Choice of the Curve}}
Let $P_{j}$ denotes the set of places of degree $j$ and $P_{j}[~k~]$ be the $k^{th}$ places of degree $j$. In order to find the suitable curve, one just have to execute the procedure below for each curve of  Theorem~\ref{BBT}:
\begin{enumerate}
\item construct the associated elliptic function field,
\item determine all places of degree 1,\,2,\,3, 4 and 5,
\item find all combinations of the divisor $ G = u_{1}P_{1}+ \cdots + u_{N}P_{N} $ with the appropriate degree,
\item for each combination, compute $\sum_{i=1}^{N} \mu_{q}(\deg P_{i},u_{i}) $ and store the lowest bilinear complexity.
\end{enumerate} 

\noindent
Note that supersingular curves can be used with no danger since we only use points for interpolation. Results of the previous procedure are collected in Table~\ref{mu357}.

\begin{table}[h]
\begin{flushleft}
\caption{Choice of the curve for $\mu^{sym}_{3}(57)$.}\label{mu357}
~~\\
\begin{tabular}{|l|c|c|c|}
\hline
 Equation  & N & U & $\mu^{sym}_{3,\mathcal{C}}(57)$ \\
\hline
$ \mathcal{C}:=  y^2 + 2x^3 + 2x^2 + 2 = 0 $  & $ [ 6, 3, 4, 21, 0 ]$ & $[ 2, 1, 1, 1, 1 ]$ & $240$ \\
\hline
$ \mathcal{C}:=  y^2 + 2x^3 + x^2 + 1 = 0 $  & $[ 2, 5, 12, 15, 1 ]$ & $[ 2, 1, 1, 1, 1 ]$ & $240$ \\
\hline
$ \mathcal{C}:=  y^2 + 2x^3 + x^2 + 2 = 0 $  & $[ 5, 5, 5, 15, 3 ]$ & $[ 3, 1, 1, 1, 1 ]$ & $241$ \\
\hline
$ \mathcal{C}:= y^2 + 2x^3 + 2x^2 + 1 = 0 $  & $[ 3, 6, 11, 15, 0 ] $& $[ 3, 1, 1, 1, 1 ]$ & $\textbf{234}$ \\
\hline
$ \mathcal{C}:=  y^2 + 2x^3 + 2x = 0 $  & $[ 4, 6, 8, 9, 6 ]$ & $[ 3, 1, 1, 1, 1 ]$  & $239$ \\
\hline
$ \mathcal{C}:= y^2 + 2x^3 + x + 2 = 0 $  & $[ 7, 0, 7, 18, 0 ] $& $[ 3, 1, 1, 1, 1 ]$ & $239$ \\
\hline
$ \mathcal{C}:=  y^2 + 2x^3 + x + 1 = 0 $  & $[ 1, 3, 9, 19, 1 ]]$ & $[ 3, 1, 1, 1, 1 ]$  & $251$ \\
\hline
\end{tabular}

\end{flushleft}
\end{table}

\noindent
From Table~\ref{mu357}, the suitable curve, up to isomorphism, is $$ E : y^2 + 2x^3 + 2x^2 + 1 = 0,$$
and the divisor $G$ is constructed as follows: we take all 3 points of degree 1 with multiplicity 3, and then we take all 6 points of degree 2, all 11 points of degree 3, and all 15 points of degree 4,  all with multiplicity 1. It must be verified that G has degree
$$ \deg G = 3.1.3 + 6.2.1 + 11.3.1 + 15.4.1 = 114 = 2 \cdot 57. $$

\noindent
Using values of Table~\ref{T3} we obtain\\

\begin{tabular}{cl}
  $ \mu^{sym}_{3}(57) $ & $ \leq 3.\mu_{3}(1,3) + 6.\mu_{3}(2,1) + 11.\mu_{3}(3,1) + 15.\mu_{3}(4,1)$\\
                  & $ \leq 3.\widehat{M_3}(3) + 6.\mu_{3}(2) + 11.\mu_{3}(3) + 15.\mu_{3}(4)$\\
  $ \mu^{sym}_{3}(57) $  & $ \leq 234$. \\
\end{tabular}

\subsubsection{\textbf{Place Q and Divisor D}}

In the following, we use the notation of magma~\cite{MAGMA} for the representation of places and divisors. In order to construct $\F_{3^{57}}$ we choose the place Q defined by \\\\
\noindent
$Q:=(x^{57} +   x^{56} +   2x^{54} +   2x^{53} +   2x^{51} +   2x^{50} +   2x^{49} +   x^{48} +   x^{46} +   x^{43} +   2x^{42} +   2x^{41} +   2x^{39} +   2x^{38} +   2x^{37} +   2x^{36} +   x^{35} +   2x^{32} +   2x^{29} +   x^{28} +   x^{27} +   2x^{26} +   x^{25} +   x^{24} +   2x^{23} +   2x^{21} +   2x^{20} +   x^{19} +   x^{18} +   2x^{15} +   x^{14} +   2x^{13} +   x^{10} +   2x^{8} +   x^{7} +   x^{6} +   2x^{5} +   x^{4} +   x^{3} +   2x^{2} +   x +   2, z +   2x^{56} +   x^{55} +   x^{54} +   x^{53} +   x^{52} +   2x^{50} +   2x^{49} +   x^{48} +   2x^{47} +   2x^{45} +   2x^{43} +   2x^{42} +   2x^{41} +   2x^{38} +   2x^{37} +   2x^{36} +   2x^{35} +   x^{34} +   x^{33} +   x^{32} +   2x^{31} +   2x^{29} +   x^{28} +   2x^{25} +   2x^{24} +   x^{23} +   2x^{22} +   2x^{20} +   x^{19} +   2x^{18} +   x^{17} +   x^{15} +   2x^{13} +   2x^{12} +   x^{11} +   x^{10} +   x^{8} +   x^{6} +   2x^{5} +   x^{2} +   2x +   1),$\\

\noindent
and we choose the following divisor $\mathscr{D}$ such that\\\\
\noindent
$ \mathscr{D} = (x^{57} +  x^{55} +  x^{53} +  x^{48} +  x^{46} +  2x^{45} +  2x^{43} +  2x^{42} +  x^{40} +  2x^{36} +  x^{35} +  x^{34} +  x^{33} +  x^{32} +  x^{29} +  2x^{27} +  x^{26} +  2x^{24} +  2x^{23} +  2x^{21} +  2x^{18} +  2x^{17} +  x^{16} +  2x^{13} +  x^{12} +  2x^{10} +  2x^{9} +  x^{8} +  2x^{7} +  2x^{6} +  2x^{3} +  2x^{2} +  x +  2, z +  x^{56} +  2x^{55} +  x^{54} +  x^{53} +  2x^{52} +  x^{51} +  x^{50} +  2x^{49} +  x^{48} +  2x^{47} +  2x^{46} +  2x^{45} +  x^{43} +  2x^{42} +  2x^{41} +  2x^{39} +  x^{38} +  x^{37} +  x^{36} +  2x^{35} +  2x^{34} +  x^{32} +  2x^{30} +  2x^{29} +  2x^{28} +  x^{27} +  x^{26} +  x^{25} +  x^{24} +  x^{21} +  x^{20} +  2x^{17} +  x^{16} +  x^{13} +  2x^{12} +  x^{10} +  x^{9} +  x^{8} +  2x^{7} +  2x^{6} +  2x^{5} +  x^{4} +  2x^{2}).$\\

\noindent
to construct $ \mathscr{B} = \{ f_{1},\ldots,f_{114} \} $ the basis of $\mathscr{L}(2\mathscr{D})$ containing a basis of $\mathscr{L}(\mathscr{D})$.

\subsubsection{\textbf{Interpolation Phase}}
 
In order to construct the effective algorithm of multiplication  in $ \F_{3^{57}}$,  namely explicit formulas for bilinear multiplications, we have to evaluate the relation (\ref{Relation}) at all points chosen to obtain the bound 234. We classify the interpolation phase starting with places used with derivative evaluation $u > 1$, and we finish by the ones used with no derivative evaluation.

\begin{itemize}

\item {\textbf{Derivative Evaluation}}\\
Remember that the higher multiplicity $ u = 3 $, occurs only with places of degree 1. This means that we use the local expansion at order 3 for all points of degree 1, hence for any function $f_{i}$ of the basis $\mathscr{B}$ we have
\begin{equation}\label{DEP1}
f_{i}(P_{1}[k]) = \alpha_{i,0} + \alpha_{i,1}t_{k} + \alpha_{i,2}{t_{k}}^2,
\end{equation}
$\mbox{where}~\alpha_{i,j}~\mbox{is an element of}~\F_{3},~\mbox{and}~t_{k}~\mbox{is the local parameter for}~P_{1}[k]$. Evaluating the relation (\ref{Relation}) at points of degree 1 leads to
\begin{equation}\label{RP1}
  \left( \sum_{i=1}^{57}a_{i}f_{i}(P_{1}[~k~])\right).\left( \sum_{i=1}^{57}b_{i}f_{i}(P_{1}[~k~])\right)= \sum_{i=1}^{114}c_{i}f_{i}(P_{1}[~k~]),
  \end{equation} 
where $k~\in~[1,\ldots,3]$, $a_{i},~b_{i},~\mbox{and}~c_{i}~\in~\F_{3}$. Substituting  expression (\ref{DEP1}) in equation (\ref{RP1}) allows us to write   

\begin{equation}\label{SP1}
\left( A_{0} + A_{1}t_{k} + A_{2}{t_{k}}^2 \right).\left( B_{0} + B_{1}t_{k} + B_{2}{t_{k}}^2 \right) =   C_{0} + C_{1}t_{k} + C_{2}{t_{k}}^2,
\end{equation}
where $$ A_{\ell} = \sum_{i=1}^{57}a_{i}\alpha_{i,\ell},~B_{\ell} = \sum_{i=1}^{57}b_{i}\alpha_{i,\ell}~~\mbox{and}~~  C_{\ell} = \sum_{i=1}^{114}c_{i}\alpha_{i,\ell}.$$

\noindent
The quantity (\ref{SP1}) is exactly the complexity of 3-multiplication of two 3-$term$ polynomials of $ \F_{3^{\deg P_1}}[t_{k}]$. We have $\widehat{M_3}(3)= 5$, meaning that to obtain the three first coefficients of the product, we need the 5 bilinear multiplications in $ \F_{3^{\deg P_1}}$
\begin{center}
\begin{tabular}{ll}
$m_{1} =$ & $A_{0}.B_{0},$ \\ 
$m_{2} =$ & $A_{1}.B_{1},$ \\ 
$m_{3} =$ & $A_{2}.B_{2},$ \\ 
$m_{4} =$ & $(A_{0}+ A_{1}).(B_{0} + B_{1}),$ \\ 
$m_{5} =$ & $(A_{0}+ A_{2}).(B_{0} + B_{2}).$
\end{tabular}
\end{center}
 
\noindent 
Remember, if we use derivative evaluation with places of degree more than one, we should have 5 bilinear multiplications in $ \F_{3^{\deg P}}$, and finally we should add $ \mu_3(\deg P)$ the bilinear complexity of multiplication in $ \F_{3^{\deg P}}$. This being said, for our example we use all 3 points of degree 1 with multiplicity 3, so we obtain 15 bilinear multiplications, which matrix representation is

\[ \left(
  \begin{array}{ c }
     m_{1}\\
     m_{4} - m_{1} - m_{2}\\
     m_{5} - m_{3} - m_{1} + m_{2}\\
m_{6}\\
     m_{9} - m_{6} - m_{7}\\
     m_{10} - m_{8} - m_{6} + m_{7}\\
m_{11}\\
     m_{14} - m_{11} - m_{12}\\
     m_{15} - m_{13} - m_{11} + m_{12}\\
     
  \end{array} \right)=
   \left(
  \begin{array}{ c }
     C_{1}\\
     C_{2}\\
     C_{3}\\
     C_{4}\\
     C_{5}\\
     C_{6}\\
     C_{7}\\
     C_{8}\\
     C_{9}\\
  \end{array} \right).
\]

\vspace{0,2cm}
\noindent
For places of higher degree, we use all of them with multiplicity 1, thus with no derivative evaluation.\\

\item {\textbf{No Derivative Evaluation}}\\

\noindent
Evaluating the relation (\ref{Relation}) at points of degree $ \deg P_j$ leads to
 
\begin{equation}\label{REval}
 \left( \sum_{i=1}^{57}a_{i}f_{i}(P_{j}[~k~])\right).\left( \sum_{i=1}^{57}b_{i}f_{i}(P_{j}[~k~])\right)=  \sum_{i=1}^{114}c_{i}f_{i}(P_{j}[~k~]).
\end{equation} 

\noindent
For any function $ f_{i} $ of the basis $\mathscr{B}$, $~f_{i}(P_{j}[~k~])$ is an element of the finite field $ \F_{3^{\deg P_j}} $ in which a representation is 
$$  \F_{3^{\deg P_j}} = \F_{3}(w_{k}) = \frac{\F_{3}[X]}{< P_{j}[~k~]>}.$$
If the set $ \{ 1,w_{k},\ldots,{w_{k}}^{j-1} \}$ denotes a basis of $ \F_{3^{\deg P_j}} $, then there exist $j$ elements, $s_{i,0},s_{i,1},\ldots,s_{i,j-1}$ of
$\F_{3}$ such that 

\begin{equation}\label{DevCo}
f_{i}(P_{j}[~k~]) = s_{i,0}w_{k}^{0} + s_{i,1}w_{k}^{1} + \cdots + s_{i,j-1}{w_k}^{j-1}. 
\end{equation}  \\

\noindent
Equation (\ref{DevCo}) allows us to rewrite relation (\ref{REval}) as  

\begin{equation}\label{mudegPj}
 \left( \sum_{\ell = 0}^{j-1}A_{\ell}{w_{k}}^{\ell} \right).\left( \sum_{\ell = 0}^{j-1}B_{\ell}{w_{k}}^{\ell} \right) =   \left( \sum_{\ell = 0}^{j-1}C_{\ell}{w_{k}}^{\ell} \right), 
\end{equation}
   
\noindent
where

$$ A_{\ell} = \sum_{i=1}^{57}a_{i}s_{i,\ell},~~~  B_{\ell} = \sum_{i=1}^{57}b_{i}s_{i,\ell},~\mbox{and}~~  C_{\ell} = \sum_{i=1}^{114}c_{i}s_{i,\ell}. $$
  
\noindent
One can easily identify expression (\ref{mudegPj}) as the multiplication of two elements of $\F_{3^{\deg P_j}}$ over $\F_{3}$. The bilinear complexity of multiplication is, in the case of interpolation at places with no derivative evaluation, $\mu_3(\deg P)$.

\begin{itemize}
 \item When $\deg P = 2$ equation (\ref{mudegPj}) becomes 
 $$ \left( \sum_{\ell = 0}^{1}A_{\ell}{w_{k}}^{\ell} \right).\left( \sum_{\ell = 0}^{1}B_{\ell}{w_{k}}^{\ell} \right) =   \left( \sum_{\ell = 0}^{1}C_{\ell}{w_{k}}^{\ell} \right), $$
and this expression is the multiplication of two elements of $\F_{3^{2}}$ over $\F_{3}$ which bilinear complexity $\mu_3(2)$ equals 3. It means that to obtain coefficients $C_0,C_1$, one need three bilinear multiplications, obtained with Karatsuba algorithm and defined by

\vspace{0,1cm}

\begin{flushleft}
\begin{tabular}{ll}
$m_{1} =$ & $A_{0}.B_{0},$ \\ 
$m_{2} =$ & $A_{1}.B_{1},$ \\ 
$m_{3} =$ & $(A_{0}+ A_{1}).(B_{0} + B_{1}).$ \\ 
\end{tabular}
\end{flushleft}

\vspace{0,1cm}

\item For degrees 3 places, we have $\mu_3(3)=6$ where the 6 multiplications needed are
\vspace{0,1cm}
\begin{flushleft}
\begin{tabular}{ll}
$ m_{1 } = $ & $ A_{0}.B_{0}, $ \\
$ m_{2 } = $ & $ A_{1}.B_{1}, $ \\
$ m_{3 } = $ & $ A_{2}.B_{2}, $ \\
$ m_{4 } = $ & $ (A_{0} + A_{1}).(B_{0} + B_{1}), $ \\
$ m_{5 } = $ & $ (A_{0} + A_{2}).(B_{0} + B_{2}), $ \\
$ m_{6 } = $ & $ (A_{1} + A_{2}).(B_{1} + B_{2}). $ \\
\end{tabular}
\end{flushleft}

\vspace{0,2cm}

\item Finally for degrees 4 places where $
_3(4)=9$, with

\vspace{0,2cm}

\begin{flushleft}
 \begin{tabular}{ll}
$ m_{1} = $ & $ A_{0}.B_{0}, $ \\
$ m_{2} = $ & $ A_{1}.B_{1}, $ \\
$ m_{3} = $ & $ A_{2}.B_{2}, $ \\
$ m_{4} = $ & $ A_{3}.B_{3} $, \\
$ m_{5} = $ & $ (A_{0} + A_{1}).(B_{0} + B_{1}), $ \\
$ m_{6} = $ & $ (A_{0} + A_{2}).(B_{0} + B_{2}), $ \\
$ m_{7} = $ & $ (A_{2} + A_{3}).(B_{2} + B_{3}), $ \\
$ m_{8} = $ & $ (A_{1} + A_{3}).(B_{1} + B_{3}), $ \\
$ m_{9} = $ & $ (A_{0} + A_{1} + A_{2} + A_{3}).(B_{0} + B_{1} + B_{2} + B_{3}). $ 
\end{tabular}
\end{flushleft}
\end{itemize}

\end{itemize}

\subsubsection{\textbf{Evaluation at Place $Q$}}
In order to complete the multiplication algorithm, we have to reconstruct $f_{\gamma}$ and then evaluate it at the chosen place $Q$. The final matrix representation of the interpolation relation $\mathscr{R}$ is

\vspace{0,1cm}

\[ \underbrace{\left(
  \begin{array}{ c }
     m_{1}\\
     m_{4} - m_{1} - m_{2}\\
     m_{5} - m_{3} - m_{1} + m_{2}\\          
     m_{6}\\
     m_{9} - m_{6} - m_{7}\\
     m_{10} - m_{8} - m_{6} + m_{7}\\     
     m_{11}\\
     m_{14} - m_{11} - m_{12}\\
     m_{15} - m_{13} - m_{11} + m_{12}\\    
     \vdots\\
     \vdots\\
     \vdots\\
 m_1 + \cdots + 2m_{230} - m_{231} + m_{234}
  \end{array} 
  \right) }_{M} =
 \underbrace{ \left(
  \begin{array}{ c }
     11101101\ldots01020010\\
     01001110\ldots10200110\\
     21201101\ldots01020022\\          
     21201101\ldots10221122\\
     11201011\ldots02021010\\
     11201110\ldots11020011\\   
     01201101\ldots01020112\\
     11201001\ldots01020211\\     
     \vdots\\
     \vdots\\
     \vdots\\    
     21201110\ldots01020011\\
     01201111\ldots01020011     
  \end{array} 
  \right) }_{G}
  \underbrace{ \left(
  \begin{array}{ c  }
     c_{1}\\
      \vdots\\
      \vdots\\
      \vdots\\
      \vdots\\
      \vdots\\
      \vdots\\
      \vdots\\              
      \vdots\\
      c_{114}
  \end{array} 
  \right). }_{C}
  \]

\vspace{1cm}
\noindent
Since $G$ is invertible, we have $ G^{-1}.M = (c_{1},\ldots,c_{114}) $ and then $ f_\gamma $ the only function of $\mathscr{L}(2\mathscr{D})$ such that $~f_{\alpha}.f_{\beta} = f_{\gamma}$ is defined by 
                            $$ f_{\gamma} = c_{1}f_{1} + \cdots + c_{114}f_{114}.$$
Recall that to obtain the product $ \alpha.\beta $ we just have to evaluate $ f_{\gamma} $ at the place $Q$ so
$$  \alpha.\beta = f_{\gamma}(Q) = c_{1}f_{1}(Q) + \cdots + c_{114}f_{114}(Q). $$

\vspace{1cm}

\subsubsection{\textbf{Reconstruction of  $ f_\gamma $ in $ \mathscr{L}(\mathscr{D}) $}}
\noindent
To complete the algorithm, we must find coefficients $\widehat{c}_i$ for $i~[1..57]$ such that 

$$   \left( \sum_{i=1}^{57}a_{i}f_{i}\right).\left( \sum_{i=1}^{57}b_{i}f_{i}\right)=  \sum_{i=1}^{57} \widehat{c}_{i}f_{i}. $$
\noindent
Let
\begin{center}
\begin{tabular}{cc}
$e_1 $ & $ = f_1(Q), $\\
$\vdots$ & $\vdots $ \\
$e_{57} $ & $ = f_{57}(Q), $\\
$\vdots$ & $\vdots $ \\
$e_{114} $ & $ = f_{114}(Q). $\\
\end{tabular}
\end{center}
\noindent
With these notations we have
                   
$$ f_{\gamma}(Q) = c_{1}e_{1} + \cdots + c_{57}e_{57} + c_{58}e_{58} + \cdots + c_{114}e_{114}.$$

\vspace{0,5cm}
\noindent
Vectors $ (e_1,\ldots,e_{57}) $ form a basis of $\F_{3^{57}}$ as $ (f_1,\ldots, f_{57}) $ is a basis of $ \mathscr{L}(\mathscr{D}) $, then to find coefficients $\widehat{c}_i$ for $i~\in~[1..57]$, it is sufficient to express vectors $ \left( e_{58},\ldots,e_{114} \right)$ according to $\left( e_1,\ldots,e_{57} \right)$. This leads to

\begin{center}
\begin{tabular}{cc}
$e_{58} $ & $ = e_1 + 2e_2 + \cdots + e_{57}, $\\
$\vdots$ & $\vdots $ \\
$e_{114} $ & $ = 2e_{1} + e_2 + \cdots + 2_{57}, $\\
\end{tabular}
\end{center}

\noindent
and bringing together terms in $\left( e_1,\ldots,e_{57} \right)$, we finally get

$$ f_{\gamma}(Q) =\underbrace{(c_{1} + c_{58}+\cdots+c_{114})}_{\widehat{c}_1}e_{1} + \cdots + \underbrace{(c_{57} + \cdots + 2c_{114}) }_{\widehat{c}_{57}}e_{57}.$$
\noindent
Explicit formulas can be found at the following address\\ http://eriscs.esil.univmed.fr/dotclear/public/res/mtukumuli/FE.pdf.\\


\end{document}